\documentclass[11pt,oneside]{article}
\usepackage{eurosym}
\usepackage{amsmath,amssymb,amsthm,amsfonts}
\usepackage[mathscr]{eucal}
\usepackage{graphicx}
\usepackage{scalerel}
\usepackage{epstopdf}
\usepackage[comma,authoryear]{natbib}
\usepackage{mathtools}
\usepackage{graphicx, subcaption}
\usepackage{lscape}
\usepackage{setspace}
\usepackage{epsfig}
\usepackage{amsmath}
\usepackage{amsfonts}
\usepackage{hyperref}
\usepackage{amssymb}

\providecommand{\U}[1]{\protect\rule{.1in}{.1in}}
\numberwithin{equation}{section} 

\newtheorem{theorem}{Theorem}[section]

\newtheorem{definition}{Definition}[section]

\newcommand{\ba}{\begin{eqnarray}}
\newcommand{\ea}{\end{eqnarray}}

\newcommand{\ibs}{\boldsymbol{i}}

\newcommand{\lambdabs}{\boldsymbol{\lambda}}

\def\AA{\boldsymbol{A}}

\def\TT{\boldsymbol{T}}

\def\YY{\boldsymbol{Y}}

\def\zz{\boldsymbol{z}}

\DeclarePairedDelimiterX{\inp}[2]{\langle}{\rangle}{#1, #2}

\begin{document}

\title{On non-linear dependence of multivariate subordinated L\'evy processes}
\author{E. Di Nardo\thanks{Department of Mathematics \lq\lq G. Peano\rq\rq, University
of Turin, elvira.dinardo@unito.it}, M. Marena\thanks{Department of Economics
and Statistics, University of Turin, marina.marena@unito.it}, P.
Semeraro\thanks{Department of Mathematical Sciences \lq\lq G. Lagrange\rq\rq, Politecnico di Torino, patrizia.semeraro@polito.it} }
\date{}
\maketitle

\begin{abstract}
Multivariate subordinated L\'{e}vy processes are widely employed in finance for modeling multivariate asset returns. We propose to exploit non-linear dependence among financial assets through multivariate cumulants of these processes, for which we provide a closed form formula by using the multi-index generalized
Bell polynomials.  Using multivariate cumulants, we perform a sensitivity analysis,  to investigate non-linear dependence as a function of the model parameters driving the dependence structure.\smallskip\newline


\end{abstract}

\section{Introduction}

L\'{e}vy processes are widely used in finance to model asset returns being more versatile than Gaussian driven processes as they can model skewness and
excess kurtosis. Their characteristic function describes the distribution of
each independent increment through the L\'{e}vy-Khintchine representation. In
the following we focus our attention on the moment generating function (mgf)
and its relation with the cumulant generating function (cgf). In particular,
if $\{\boldsymbol{L}(t),t\geq0\}$ is a $\mathbb{R}^{n}$-valued L\'{e}vy
process with mgf $M_{\boldsymbol{L}(t)}(\boldsymbol{z}),\boldsymbol{z}%
\in\mathbb{R}^{n}$ at each $t,$ the L\'{e}vy-Khintchine
representation allows us to work with $\exp\big[tK_{\boldsymbol{Y}
}(\boldsymbol{z})\big],$ where $K_{\boldsymbol{Y}}(\boldsymbol{z})$ is the cgf of $\boldsymbol{Y}%
=\boldsymbol{Y}(1),$ the time one distribution of the L\'{e}vy process.

Stochastically altering the clock on which the L\'{e}vy process is run models
the economic time of the overall market activity: time runs fast when there
are a lot of orders, while it slows down when trade is stale. The
subordination of a L\'{e}vy process $\boldsymbol{L}(t)$ by a univariate
subordinator $T(t)$, i.e. a L\'{e}vy process on $\mathbb{R}_{+}=[0,\infty)$
with increasing trajectories, independent of $\boldsymbol{L}(t)$, defines a
new process $\boldsymbol{Y}(t)$ by the composition $\boldsymbol{Y}%
(t):=\boldsymbol{L}(T(t)).$ Unfortunately, the resulting models exhibit
several shortcomings including the lack of independence between asset returns
and a limited span of linear correlations. Furthermore, there is empirical
evidence that trading activity is different across assets (\cite{Ha}). From
the theoretical perspective, multivariate subordination allowing different
assets to have different time-changes was introduced in the work of \cite{Ba}.
Given a $\mathbb{R}^{n}$-valued multiparameter L\'{e}vy process
$\{\boldsymbol{L}(\boldsymbol{s}),\,\boldsymbol{s}\in\mathbb{R}_{+}^{d}\}$ as
defined in \cite{Ba}, the $\mathbb{R}^{n}$-valued subordinated L\'{e}vy
process $\boldsymbol{Y}(t)$ is the composition $\boldsymbol{Y}%
(t)=\boldsymbol{L}(\boldsymbol{T}(t)),$ where $\boldsymbol{T}(t)$ is a
multivariate subordinator, i.e. a L\'{e}vy process on $\mathbb{R}_{+}^{d}$
whose trajectories are increasing in each coordinate, independent of
$\boldsymbol{L}(t)$. 

In Section $2,$ we give a closed form formula of joint
(or cross) cumulants of $\boldsymbol{Y}(t)$ through the multi-index
generalized (complete exponential) Bell polynomials introduced in
\cite{di2011new}. We use an \textit{umbral} evaluation operator
(\cite{di2015symbolic}) to recover the contribution of joint cumulants of the $d$-dimensional subordinator. For multivariate subordinated Brownian motions,
this closed form formula further simplifies by taking advantage of the
well-known property that cumulants of Brownian motion are zero when their
order is greater than two. The case of subordinated Brownian motions is of
particular interest in finance for two reasons: first they link deviation of
normality of asset returns to trade activity and second they often have
analytical characteristic functions. A subclass of multivariate subordinated Brownian motions widely used in
finance because of their economic interpretation is the so called $\rho\alpha
$-models, see \cite{LuciSem1}. These models exhibit a flexible dependence
structure and allow to model also high correlations. They also incorporate
non-linear dependence as discussed in \cite{LuciSem1}: however this feature
has not be investigated so far. 

In Section $3,$ we propose to use joint cumulants to study higher order
dependence and its behaviour in time. Indeed, as well known, the covariance matrix completely describes the dependence structure among components of a multivariate process only for Gaussian ones. Co-skewness and co-kurtosis measure extreme deviations or dispersions undergone by the components, as in the univariate case they measure asymmetry and fat-tailedness. Moreover, higher order cumulants play an important role in the
analysis of non-Gaussian data and allow to detect higher order
cross-correlations, a critical feature that exacerbates during financial
turmoils, see \cite{domino2018efficient}. If joint cumulants are
asymptotically zero, central limit theorems can be investigated to forecast
the market behavior or viceversa to choose among different models the one
which better incorporates non-linear dependence. As case study, we investigate
the non-linear dependence structure of $\rho\alpha$-models describing asset
returns as a superposition of an idiosyncratic component, due to the asset
specific trades, and a systematic one, due to the overall trade. We focus our
attention on the Normal Inverse Gaussian (NIG) specification, whose one
dimensional marginals are NIG processes and discuss the role played by the model
parameters in driving non-linear dependence.


\section{Cumulants of multivariate subordinated L\'{e}vy processes}

Let us consider the multiparameter L\'{e}vy process $\boldsymbol{L}%
(\boldsymbol{s})=\boldsymbol{A}\boldsymbol{Z}(\boldsymbol{s})$ where
$\boldsymbol{Z}(\boldsymbol{s})=(Z_{1}(s_{1}),\ldots,Z_{d}(s_{d})),$
$\boldsymbol{s}\in\mathbb{R}_{+}^{d},$ is a multiparameter L\'{e}vy process
with independent components and $\boldsymbol{A}=(a_{ij})_{n\times d}%
\in{\mathbb{R}}^{n\times d}.$ According to \cite{Ba}, the subordinated process
$\boldsymbol{Y}(t)$
\begin{equation}
\boldsymbol{Y}(t):=\boldsymbol{A}\boldsymbol{Z}(\boldsymbol{T}%
(t))=\boldsymbol{A}(Z_{1}(T_{1}(t)),\cdots,Z_{d}(T_{d}(t)))^{T}\label{BMsub}%
\end{equation}
is a L\'{e}vy process. It's a straightforward consequence of Theorem 4.7 in
\cite{Ba} to prove that
\begin{equation}
K_{\boldsymbol{Y}}(\boldsymbol{z})=K_{\boldsymbol{T}}\left(  K_{Z_{1}%
}\left(\sum_{m=1}^{n}a_{m1}z_{m}\right),\ldots,K_{Z_{d}}\left(\sum_{m=1}%
^{n}a_{md}z_{m}\right)\right)  \!\!,\,\boldsymbol{z}\in\mathbb{R}%
^{n}\!\!.\label{MGFgen}%
\end{equation}
Indeed if $\boldsymbol{\delta}_{j}=(\delta_{j1},\ldots,\delta_{jd})$ with
Kronecker's $\delta_{jk},$ then $\boldsymbol{L}(\boldsymbol{\delta}_{j})=(a_{1j}Z_{j}(1),\ldots,a_{nj}Z_{j}(1))^{T}$ has mgf
\begin{equation}
M_{j}(\boldsymbol{z})=E\left[  \exp\left(  Z_{j}(1)\sum_{m=1}^{n}a_{mj}%
z_{m}\right)  \right]  =M_{Z_{j}}\left(  \sum_{m=1}^{n}a_{mj}z_{m}\right)
.\label{momZ}%
\end{equation}
Since $\log M_{Z_{j}}$ is the cgf of ${Z_{j}}(1)$ and $K_{\boldsymbol{Y}%
}(\boldsymbol{z})=\log M_{\boldsymbol{Y}}(\boldsymbol{z}),$ then
(\ref{MGFgen}) follows by plugging (\ref{momZ}) in $M_{\boldsymbol{Y}%
}(\boldsymbol{z})=\exp[K_{\boldsymbol{T}}\left(\log M_{1}(\boldsymbol{z}%
),\ldots,\log M_{d}(\boldsymbol{z})\right)].$ In order to recover the
$\ibs$-th cumulant $c_{\boldsymbol{i}}(\boldsymbol{Y})$  of $\boldsymbol{Y}(1),$ we need to expand in formal power series the cgf $K_{\boldsymbol{Y}%
}(\boldsymbol{z})$ given in (\ref{MGFgen}). To this aim, let us recall the
notion of multi-index partition introduced in \cite{di2011new}.

\begin{definition}
\label{multpart1} A partition of a multi-index ${\boldsymbol{i}} = (i_{1},
\ldots, i_{n}) \in{\mathbb{N}}^{n}$ is a matrix $\Lambda$ of non-negative
integers with $n$ rows and no zero columns in lexicographic order, such that
$\lambda_{s 1}+\lambda_{s 2}+\cdots=i_{s}$ for $s = 1,2,\ldots,n.$
\end{definition}

\vskip-0.2cm As for integer partitions, let us fix some notation:\newline%
\indent
$\bullet$ $|\Lambda|$ is the sum of all components of $\Lambda$\newline%
\indent
$\bullet$ $\Lambda=(\boldsymbol{\lambda}_{1}^{r_{1}}, \boldsymbol{\lambda}
_{2}^{r_{2}}, \ldots) \vdash{\boldsymbol{i}}$ denotes the multi-index
partition of ${\boldsymbol{i}}$ with $r_{1}$ columns equal to
$\boldsymbol{\lambda}_{1},$ $r_{2}$ \newline\indent columns equal to
$\boldsymbol{\lambda}_{2}$ and so on, with $\boldsymbol{\lambda}_{1} <
\boldsymbol{\lambda}_{2} < \ldots$\newline\indent
$\bullet$ $\mathfrak{m}(\Lambda)=(r_{1}, r_{2}, \ldots)$ is the vector of
multiplicities of $\boldsymbol{\lambda}_{1}, \boldsymbol{\lambda}_{2}, \ldots$
\newline\indent
$\bullet$ $l(\Lambda)=|\mathfrak{m}(\Lambda)|=r_{1} + r_{2} + \cdots$ is the
number of columns of $\Lambda$ and $\Lambda! = \prod_{j=1}^{l(\Lambda
)}(\boldsymbol{\lambda}_{j}!)^{r_{j}}$\newline\indent
$\bullet$ given the multi-indexed sequence $\{g_{\boldsymbol{j}}\},$ the
product $g_{\Lambda}= \prod_{j=1}^{l(\Lambda)}g_{\boldsymbol{\lambda}_{j}
}^{r_{j}}$ is said associated to the \newline\indent
sequence through $\Lambda\vdash{\boldsymbol{i}},$ in particular $g_{\Lambda} =
1$ if $\Lambda\vdash{\boldsymbol{i}} ={\boldsymbol{0}}.$ \smallskip \newline
The $\boldsymbol{i} $-th coefficient of $\exp \big( \sum_{k=1}^{d} x_{k}
g_{k}(\boldsymbol{z}) \big)$ with $g_{k}(\boldsymbol{z}) = \sum
_{\boldsymbol{j}: |\boldsymbol{j}| > 0} g_{k,\boldsymbol{j}} \frac
{\boldsymbol{z}^{\boldsymbol{j}}}{\boldsymbol{j} !},$ is the $\boldsymbol{i}%
$-th generalized (complete exponential) Bell polynomial (\cite{di2011new} )
\begin{equation}
{\mathfrak{B}}_{\boldsymbol{i}}(x_{1}, \ldots, x_{d}) = \boldsymbol{i}! \sum_{%
\genfrac{}{}{0pt}{}{{\Lambda} \vdash\boldsymbol{s}_{1}, \ldots, \tilde
{\Lambda} \vdash\boldsymbol{s}_{d} }{\boldsymbol{s}_{1} + \cdots+
\boldsymbol{s}_{d}=\boldsymbol{i}}%
} x_{1}^{l(\Lambda)} \cdots x_{d}^{l(\tilde{\Lambda})} \frac{g_{1,\Lambda}
\cdots g_{d,\tilde{\Lambda} }}{\Lambda! \cdots\tilde{\Lambda}! \,
\mathfrak{m}(\Lambda)! \cdots\mathfrak{m}(\tilde{\Lambda})!}
\label{(sol22ter)}%
\end{equation}
\vskip-0.2cm \noindent where $g_{k,\Lambda}$ is associated to the sequence
$\{g_{k,\boldsymbol{j}}\}$ through $\Lambda\vdash\boldsymbol{s}_{k}$ for $k=1,
\ldots,d.$

\begin{theorem}
\label{(sol22)} The $\boldsymbol{i}$-th cumulant of $\boldsymbol{Y}$ is
$c_{\boldsymbol{i}}(\boldsymbol{Y}) = {E} \left(  {\mathfrak{B}%
}_{\boldsymbol{i}}(T_{1}, \ldots, T_{d}) \right)  $ where
\begin{description}
\label{cumulants}
\item[{\it i)}] ${E}$ is an \textit{umbral} evaluation linear
operator (\cite{di2011new}) such ${E}(T_{1}^{i_{1}} \cdots
T_{d}^{i_{d}})= c_{\boldsymbol{i}}({\boldsymbol{T}}), \boldsymbol{i}
\in{\mathbb{N}}^{n},$ the $\boldsymbol{i}$-th joint cumulant of
${\boldsymbol{T}};$
\item[{\it ii)}] $g_{k,\Lambda}$ is associated to $\{g_{k,
\boldsymbol{\lambda}_{j}}\}$ with $g_{k, \boldsymbol{\lambda}_{j}} =
c_{|\boldsymbol{\lambda}_{j}|}(Z_{k}) (\boldsymbol{a}_{.k}%
)^{\boldsymbol{\lambda}_{j}}$ where $\{c_{|\boldsymbol{\lambda}_{j}|}%
(Z_{k})\}$ are cumulants of $Z_{k}$ and $\boldsymbol{a}_{.k}$ is the $k$-th
column of $\boldsymbol{A}.$
\end{description}
\end{theorem}

\begin{proof}
From (\ref{MGFgen}) the cgf of $\YY$ is a composition $f[g_1(\zz), \ldots, g_d(\zz)],$ with $g_{k}(\zz)=K_{Z_k}(\sum_{m=1}^{n}a_{mk} z_{m})$ and $f(\tilde{\zz}) = K_{\TT}(\tilde{\zz}), \tilde{\zz} = (\tilde{z}_1, \ldots, \tilde{z}_d).$
Then the $\ibs$-th coefficient of $f[g_1(\zz), \ldots, g_d(\zz)]$ is obtained by expanding in formal power series $\exp \big( \sum_{k=1}^d x_k g_k(\zz) \big)$ and then by applying the evaluation linear operator ${\mathbb E},$ see \cite{di2011new}  for the details. Thus we recover (\ref{(sol22ter)}) with the product $x_1^{l(\Lambda)} \cdots x_d^{l(\tilde{\Lambda})}$ replaced by the $(l(\Lambda), \ldots, l(\tilde{\Lambda}))$-th coefficient of the formal power series $f(\tilde{\zz}),$   that is the joint cumulant $c_{l(\Lambda), \ldots, l(\tilde{\Lambda})}({\TT}).$ By further expanding $K_{Z_k}$ in formal power series, we  have $g_{k,\boldsymbol{\lambda}}=c_{|\boldsymbol{\lambda}|}(Z_k) (\boldsymbol{a}_{.k})^{\boldsymbol{\lambda}}$ where $c_{|\boldsymbol{\lambda}|}(Z_k)$ is the $|\boldsymbol{\lambda}|$-th cumulant of $Z_k$ and $\boldsymbol{a}_{.k}$ is the $k$-th column of $\AA.$  Thus the result follows from \eqref{(sol22ter)}.
\end{proof}

\subsection{Multivariate subordinated Brownian motion}

\label{SBM}
As multiparameter L\'evy process $\boldsymbol{Z}(\boldsymbol{s})$ in
\eqref{BMsub} let us consider $\boldsymbol{B}(\boldsymbol{s}) = (B_{1}(s_{1}),
\ldots,B_{d}(s_{d}))^{\top}, \boldsymbol{s}\in\mathbb{R}^{d}_{+},$ obtained
from a $\mathbb{R}^{d}$-valued Brownian motion $\{ \boldsymbol{B}(t), t
\geq0\}$ with independent components, drift ${\boldsymbol{\mu}}$ and
covariance matrix ${\boldsymbol{\Sigma}}=\mathrm{diag}({\sigma}_{1}^{2}%
,\ldots,{\sigma}_{d}^{2}),$ and the multiparameter L\'evy process
\begin{equation}
\label{Brho}{\boldsymbol{B}_{A}}(\boldsymbol{s}):=\boldsymbol{A}
{{\boldsymbol{B}}}(\boldsymbol{s}) = \boldsymbol{A} (B_{1}(s_{1}),
\ldots,B_{d}(s_{d}))^{\top}, \,\,\, \boldsymbol{s}\in\mathbb{R}^{d}_{+}.
\end{equation}
Note that $\boldsymbol{B}_{A}(t)=\boldsymbol{A} {\boldsymbol{B}}(t)
=\boldsymbol{A} ({B}_{1}(t),\ldots,{B}_{n}(t))^{\top}$ has drift
$\boldsymbol{\mu}_{A}=\boldsymbol{A} \boldsymbol{\mu}$ and covariance matrix
$\boldsymbol{\Sigma}_{A}=\boldsymbol{A} \boldsymbol{\Sigma} \boldsymbol{A}%
^{\top}$ and that, if $n=d$ and $\boldsymbol{A}$ is the identity matrix, we
recover the subcase of independent Brownian motions. Suppose $\boldsymbol{T}
(t)$ independent of $\boldsymbol{B}_{A}(\boldsymbol{s})$ and consider
$\boldsymbol{Y}(t)=\boldsymbol{B}_{A}(\boldsymbol{T}(t)).$ Theorem
\ref{cumulants} allows us to compute the $\boldsymbol{i}$-th cumulant
$c_{\boldsymbol{i}}(\boldsymbol{Y})$ by taking advantage of the well-known
property that cumulants of a Brownian motion are zero when their order is
greater than $2.$ Indeed in (\ref{(sol22ter)}), the products $g_{k,\Lambda} =
\prod_{j=1}^{l(\Lambda)}g_{k,\boldsymbol{\lambda}_{j}} ^{r_{j}}$ are not zero
if and only if $|\boldsymbol{\lambda}_{j}| \leq2$ that is $\Lambda
\in{\mathcal{P}}_{2}(\boldsymbol{s}_{k})= \{ (\boldsymbol{\lambda}_{1}^{r_{1}%
}, \boldsymbol{\lambda}_{2}^{r_{2}}, \ldots) \vdash\boldsymbol{s}:
|\boldsymbol{\lambda}|\leq2 \},$ the set of all multi-index partitions whose
columns have sum of components not greater than $2.$ Thus the summation in
(\ref{(sol22ter)}) reduces to $\{\Lambda\in{\mathcal{P}}_{2}(\boldsymbol{s}
_{1}), \ldots, \tilde{\Lambda} \in{\mathcal{P}}_{2}(\boldsymbol{s}_{d}) \,
\mathrm{with} \, \boldsymbol{s}_{1} + \cdots+ \boldsymbol{s}_{d}
=\boldsymbol{i}\}.$ Moreover the sequence $\{g_{k,\boldsymbol{\lambda}_{j}}\}$
is given by
\begin{equation}
g_{k,\boldsymbol{\lambda}_{j}} = \left\{
\begin{array}
[c]{ll}%
a_{mj} \mu_{j} &
\hbox{if $|\lambdabs_j|=1$ and $ (\lambdabs_j)_m= 1$ for $m \in [n],$}\\
a_{mj}^{2} \sigma^{2}_{j} & \hbox{if $|\lambdabs_j|=2$ and
$(\lambdabs_j)_m= 2$ for $m \in [n],$}\\
a_{m_{1} j} a_{m_{2} j} \sigma_{j}^{2} &
\hbox{if $|\lambdabs_j|=2$ and $(\lambdabs_j)_{m_1}=(\lambdabs_j)_{m_2}=1$
for $m_1 \ne m_2 \in [n],$}\\
0 & \hbox{if $|\lambdabs_j| >2,$}
\end{array}
\right.  \label{(cum1fin)}%
\end{equation}
with $[n]=\{1, \ldots,n\}.$ Indeed in (\ref{MGFgen}) the cgf $K_{Z_{k}}%
(\sum_{m=1}^{n} a_{mk} z_{m})$ reduces to $\mu_{k}\sum_{m=1}^{n} a_{mk}z_{m}
+\frac{1}{2} \sigma_{k}^{2}\big(\sum_{m=1}^{n} a_{mk} z_{m}\big)^{2}$ as
$\{Z_{k}\}_{k=1}^{d}$ are Gaussian distributed with mean $\{\mu_{k}%
\}_{k=1}^{d}$ and variance $\{\sigma^{2}_{k}\}_{k=1}^{d}.$ Note that when
$d=1, {E} \left(  {\mathfrak{B}}_{\boldsymbol{i}}(T_{1}, \ldots,
T_{d}) \right)  $ in Theorem \ref{cumulants} reduces to
\vskip-0.3cm
\begin{equation}
c_{\boldsymbol{i}}(\boldsymbol{Y}) = \boldsymbol{i}! \sum_{\Lambda
\in{\mathcal{P}}_{2}(\boldsymbol{i})} c_{l(\Lambda)}(T) \prod_{s=1}%
^{l(\lambda)} \frac{g_{\boldsymbol{\lambda}_{s}}^{r_{s}}}{(\boldsymbol{\lambda
}_{s}!)^{r_{s}} r_{s}!}. \label{(sol228bis)}%
\end{equation}

\subsection{The $\rho\alpha$-model}

In this section, we consider a class of processes used in finance to model
asset returns: the $\rho\alpha$-models. \cite{jevtic2017Marked} proved that
they belong to the class of multivariate subordinated Brownian motions
\eqref{Brho}, by properly choosing the matrix $A$ and the subordinator
$\boldsymbol{T}(t).$

A $\rho\alpha$-model (\cite{LuciSem1}) is constructed by subordinating $n$
independent Brownian motions $B_{j}(t)$ with $n$ independent subordinators
$\{X_{j}(t)\}$ and by subordinating a multidimensional Brownian motion
$\boldsymbol{B}^{\rho}(t)$ with a unique subordinator $Z(t).$ More in details,
let $\boldsymbol{B}^{\rho}(t)=(B_{1}^{\rho}(t),...,B_{n}^{\rho}(t))$ be a
multivariate Brownian motion, with correlations $\boldsymbol{\rho}$ and
L\'{e}vy triplet $(\boldsymbol{\mu}^{\rho},\boldsymbol{\Sigma}^{\rho
},\boldsymbol{0})$, where $\boldsymbol{\mu}^{\rho}=(\mu_{1}\alpha_{1}%
,\ldots,\mu_{n}\alpha_{n}),$ with $\boldsymbol{\mu}\in\mathbb{R}^{n},$ and
$\boldsymbol{\Sigma}^{\rho}:=(\rho_{ij}\sigma_{i}\sigma_{j}\sqrt{{\alpha_{i}}%
}\sqrt{{\alpha_{j}}})_{ij}$, with $\sigma_{j}>0$ and $\rho_{ij}\in
\lbrack-1,1]$ for $i,j=1,\ldots,n.$ Assume $\boldsymbol{B}^{\rho}(t)$
independent of $\boldsymbol{B}(\boldsymbol{s})=(B_{1}(s_{1}),\ldots
,B_{n}(s_{n}))^{\top}$ obtained from a $\mathbb{R}^{n}$-valued Brownian motion
$\{\boldsymbol{B}(t),t\geq0\}$ with independent components, drift
${\boldsymbol{\mu}}$ and covariance matrix ${\boldsymbol{\Sigma}%
}=\mathrm{diag}({\sigma}_{1}^{2},\ldots,{\sigma}_{n}^{2}).$ The $\rho\alpha
$-model $\boldsymbol{Y}(t)$ is the $\mathbb{R}^{n}$-valued subordinated
process
\begin{equation}
\boldsymbol{Y}(t)=\left(
\begin{array}
[c]{c}%
Y_{1}^{I}(t)+Y_{1}^{\rho}(t)\\
\vdots\\
Y_{n}^{I}(t)+Y_{n}^{\rho}(t)
\end{array}
\right)  =\left(
\begin{array}
[c]{c}%
B_{1}(X_{1}(t))+B_{1}^{\rho}(Z(t))\\
\vdots\\
B_{n}(X_{n}(t))+B_{n}^{\rho}(Z(t))
\end{array}
\right)  ,\label{abgp}%
\end{equation}
where $X_{j}(t)$ and $Z(t)$ are independent subordinators, independent of
$\boldsymbol{B}(t)$ and $\boldsymbol{B}^{\rho}(t).$ If $\rho=0$, the process
$\boldsymbol{Y}(t)$ is said an $\alpha$-model. As $\boldsymbol{Y}^{I}$ has
independent margins, from the additivity property of cumulants, we have
$K_{\boldsymbol{Y}}(\boldsymbol{z})=K_{\boldsymbol{Y}^{I}}(\boldsymbol{z}%
)+K_{\boldsymbol{Y}^{\rho}}(\boldsymbol{z})=\sum_{j=1}^{n}g_{j}(\boldsymbol{z}%
)+f(\boldsymbol{z}),$ where
\begin{equation}
g_{j}(\boldsymbol{z})=K_{X_{j}}\left(  \mu_{j}z_{j}+\frac{1}{2}\sigma_{j}%
^{2}z_{j}^{2}\right)  \quad\mathrm{and}\quad f(\boldsymbol{z})=K_{Z}\left(
\boldsymbol{z}^{T}\boldsymbol{\mu}^{\rho}+\frac{1}{2}\boldsymbol{z}%
^{T}\boldsymbol{\Sigma}^{\rho}\boldsymbol{z}\right)  ,\label{cumFBgen}%
\end{equation}
and joint cumulants \vskip-0.3cm
\begin{equation}
c_{\boldsymbol{i}}(\boldsymbol{Y})=\left\{
\begin{array}
[c]{ll}%
\sum_{j=1}^{n}g_{j,\boldsymbol{i}}+f_{\boldsymbol{i}} &
\hbox{if $\ibs=(0,\ldots, i_m, \ldots, 0),$}\\
f_{\boldsymbol{i}} & \hbox{otherwise,}
\end{array}
\right.  \label{(cumrhoalpha)}%
\end{equation}
with $g_{j,\boldsymbol{i}}$ and $f_{\boldsymbol{i}}$ the $\boldsymbol{i}$-th
coefficient of $g_{j}(\boldsymbol{z})$ and $f(\boldsymbol{z})$ in
\eqref{cumFBgen} respectively. By using \eqref{(sol228bis)} we have
\begin{equation}
g_{j,\boldsymbol{i}}=\boldsymbol{i}!\sum_{\Lambda\vdash\boldsymbol{i}%
}c_{l(\Lambda)}(X_{j})\prod_{s=1}^{l(\Lambda)}\frac{(g_{j,\boldsymbol{\lambda
}_{s}})^{r_{s}}}{r_{s}!(\boldsymbol{\lambda}_{s}!)^{r_{s}}}\qquad
\mathrm{and}\qquad f_{\boldsymbol{i}}=\boldsymbol{i}!\sum_{\Lambda
\vdash\boldsymbol{i}}c_{l(\Lambda)}(Z)\prod_{s=1}^{l(\Lambda)}\frac{(\tilde
{g}_{\boldsymbol{\lambda}_{s}})^{r_{s}}}{r_{s}!(\boldsymbol{\lambda}%
_{s}!)^{r_{s}}}\label{(sol11)}%
\end{equation}
where $c_{l(\Lambda)}(X_{j})$ and $c_{l(\Lambda)}(Z)$ are the $l(\Lambda)$-th
cumulants of $X_{j}$ and $Z$ respectively, $\{g_{j,\boldsymbol{\lambda}_{s}%
}\}$ and $\{\tilde{g}_{\boldsymbol{\lambda}_{s}}\}$ are the coefficients of
the inner power series of $K_{X_{j}}$ and $K_{Z}$ respectively, as given in
(\ref{cumFBgen}). To discuss non-linear dependence for the bivariate case in
Section $3,$ we consider the normalized cumulants
\begin{equation}
\bar{c}_{i,j}(t)=\frac{c_{i,j}[\boldsymbol{Y}(t)]}{\left(  c_{2}\left[
Y_{1}(t)\right]  \right)  ^{i/2}\left(  c_{2}\left[  Y_{2}(t)\right)  \right]
^{j/2}},\quad i,j\in\mathbb{N}\label{normcum}%
\end{equation}
as a function of the time scale. For $i+j\leq4$ it's straightforward to get
\begin{equation}%
\begin{array}
[c]{rclrcl}%
\bar{c}_{1,1}(t) & = & a(b_{1,1}+d_{1,1}\rho_{12}) & \bar{c}_{1,2}(t) & = &
\frac{a}{\sqrt{t}}(b_{1,2}+d_{1,2}\rho_{12})\\
\bar{c}_{1,3}(t) & = & \frac{a}{t}(b_{1,3}+d_{1,3}\rho_{12}) & \bar{c}%
_{2,2}(t) & = & \frac{a}{t}(b_{2,2}+d_{2,2}\rho_{12}+e_{2,2}\rho_{12}^{2})
\end{array}
\label{cc4}%
\end{equation}
where $b_{i,j},d_{i,j}$ and $e_{i,j}$ are functions of the marginal parameters
$(\mu_{i},\sigma_{i},\alpha_{i})$. Cumulants of $\boldsymbol{Y}(t)$ increase
linearly in $t$ so that co-skewness measures are proportional to $1/\sqrt{t}%
,$ while co-kurtosis measures are proportional to $1/t,$ converging to Gaussian
values asymptotically. Notice that $a$ is a scale parameter for all cross-cumulants,
driving the general level of dependence, both linear and non-linear.
Furthermore, the Brownian motion correlation $\rho_{12},$ providing
an extra-term in (\ref{cc4}), affects non only asset correlation measured by
$\bar{c}_{1,1}(t),$ but also non-linear dependence measured by the
other cross-cumulants.  Thus the $\rho\alpha$-models not
only span a wider range of linear dependence compared to the $\alpha$-models,
but they can also incorporate higher non-linear dependencies. 

\section{A case study}\label{Ap}
\vskip-0.2cm
To show the role played by cumulants in analyzing non-linear dependence over
 time, let us consider a bivariate price process $\{\boldsymbol{S}%
(t),\,t\geq0\}$ such that $\boldsymbol{S}(t)=\boldsymbol{S}(0)\exp
(\boldsymbol{d}t+\boldsymbol{Y}(t)),$ where $\boldsymbol{d}\in\mathbb{R}^{2}$
is the drift term, not affecting the dependence structure, and $\boldsymbol{Y}%
(t)$ is a bivariate L\'{e}vy process.
Since we are in the class of multivariate L\'{e}vy models, the centered
asymptotic distribution of daily logreturns is a bivariate Normal distribution
$N(\boldsymbol{0},\Sigma)$ where $\Sigma$ is the constant covariance matrix of
the process (see \cite{jammalamadaka2004higher}). This analysis is performed considering the NIG specification of the $\rho\alpha$-model.

\subsection{NIG specification}

Recall that a NIG process $Y(t)$ has no Gaussian component, is of infinite
variation and can be constructed by subordination as $Y(t)=\beta\delta
T(t)+\delta^{2} B(T(t)),$ where $T \sim\mathrm{IG}(1, \delta\sqrt{\gamma
^{2}-\beta^{2}})$ is independent of the standard Brownian motion $B(t).$ In
particular the time one distribution $Y(1)$ has a NIG distribution
$\mathrm{NIG}(\gamma, \delta, \beta)$ with parameters $\gamma>0,
|\beta|<\gamma$ and $\delta>0$.

Now, let us consider the $\rho\alpha$-model (\ref{abgp}), with $X_{j}%
\sim\mathrm{IG}( 1-{a}\sqrt{\alpha_{j}},\alpha_{j}^{\scriptscriptstyle -1/2}%
)\, j=1,...,n$ and $Z\sim\mathrm{IG}(a,1)$ the time one distributions of the
subordinators $X_{j}(t)$ and $Z(t)$ respectively. If we choose the parameters
of $\boldsymbol{B}^{\rho}(t)$ and of the subordinators so that:
\begin{equation}
\alpha_{j}^{\scriptscriptstyle -1/2} = \delta_{j} \sqrt{\gamma_{j}^{2} -
\beta^{2}_{j}} \qquad\hbox{and} \qquad\mu_{j}=\beta_{j}\delta_{j}^{2},
\,\,\sigma_{j}=\delta_{j} \label{NIGpar}%
\end{equation}
with $\gamma_{j}, \delta_{j} \in\mathbb{R}_{+}, \beta_{j} \in\mathbb{R}$ and
$|\beta_{j}|<\gamma_{j}$ for $j=1, \ldots,n$, the subordinated process
$\boldsymbol{Y}(t)$ in \eqref{abgp} has the remarkable property that its one
dimensional marginal processes have NIG distributions NIG$(\gamma_{j},
\delta_{j}, \beta_{j})$ for $t=1.$ The process $\boldsymbol{Y}(t)$ is named
$\rho\alpha$-NIG process. Thus $c_{\boldsymbol{i}}(\boldsymbol{Y})$ is given
in \eqref{(cumrhoalpha)} for suitable replacements of the parameters in
\eqref{(sol11)}, taking into account that if $X \sim\mathrm{IG}(a,b),$ its
cumulants are $c_{1}(X)=a/b$ and $c_{k}(X)=a (2k-3)!!/b^{2k-1}$ for $k=2,3,
\ldots.$
\subsection{Numerical results}
In this section, we use the cross-cumulants \eqref{cc4} to study the evolution
in time of non-linear dependence in a bivariate $\rho\alpha$-NIG model as a
function of the common parameters $\rho\,$and $a$. 
In fact, dependence
is driven by the common subordinator $Z$ representing the systematic
component in two ways: first, through the common parameter $a$ which defines the distribution of the common time change; second, through the action of the Brownian motions' correlation $\rho$ on $Z$.  The reference parameter
set is given in \cite{jevtic2017Marked} to which we refer for further
details\footnote{The $\rho\alpha$-NIG specification has been calibrated by the
generalized method of moments (GMM) to a bivariate basket composed by Goldman
Sachs and Morgan Stanley US daily logreturns from January 3, 2011 to December
31st, 2015. Marginal parameters are: $\gamma_{1}=85.4175,\gamma_{2}%
=64.2544,\delta_{1}=0.0248,\delta_{2}=0.0335,\beta_{1}=-8.8886,\delta
_{2}=-13.5988$. The marginal drifts not involved in the dependence structure
are $d_{1}=0.0027$ and $d_{2}=0.0074$}. We plot the normalized cumulants
$\bar{c}_{i,j}(t)$ in \eqref{normcum} up to the fourth order as a function of
the Brownian correlation $\rho$. We consider two scenarios corresponding to different values of
the scaling parameter $a$. Since $a\in(0,2.1)$ and the limit value $a=0$
corresponds to independence, we choose the intermediate value $a=1.05$ and the
higher boundary value $a=2.1$. For each scenario, we plot the evolution of
cross-cumulants for three levels of time to maturity. Comparing the two
scenarios, it is evident that $a$ drives both linear and non-linear dependence
and it allows to reach maximal correlation. Nevertheless, also in the first
scenario, where the maximal attainable asset correlation is $0.5,$ moving
$\rho,$ we are able to incorporate non-linear dependence. The evolution in
time confirms that higher order cumulants go to zero according to the rates in
equation \eqref{cc4}. \begin{figure}[ptbh]
\centering
\includegraphics[scale=0.70]{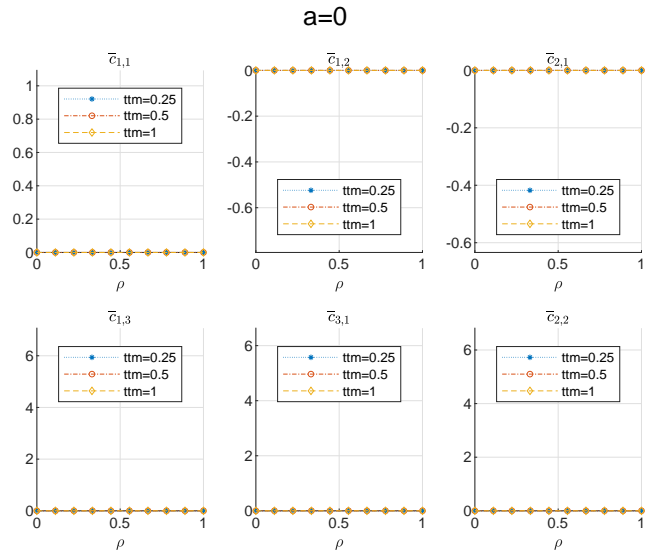}\caption{{\protect\footnotesize {Scenario
1. Normalized cross-cumulants.}}}%
\label{Mmoments}%
\end{figure}\begin{figure}[ptbh]
\centering
\includegraphics[scale=0.70]{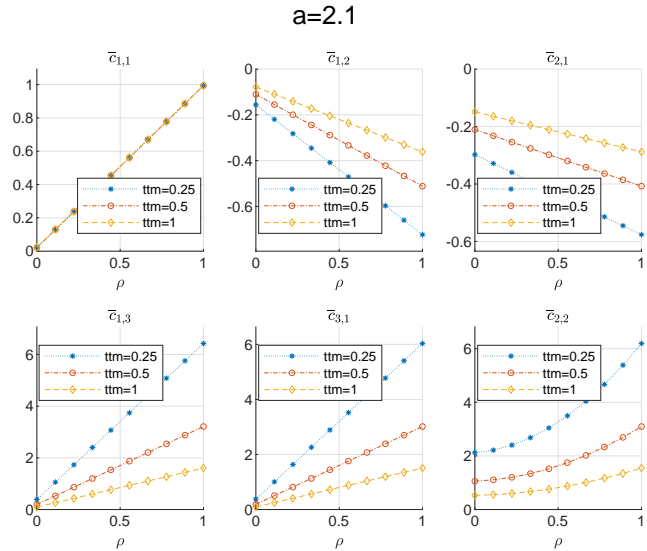}
\caption{{\protect\footnotesize {Scenario 2. Normalized cross-cumulants.} }}%
\label{moments}%
\end{figure}
\section{Conclusions}
\vskip-0.2cm
Cumulants are a powerful tool to measure non-linear dependence. We take
advantage of the L\'{e}vy-Kinthchin representation of multiparameter
subordinated L\'{e}vy processes to find their cumulants in closed form. We use
them to study non-linear dependence captured by some class of processes widely
used in finance to model asset returns. Indeed higher order statistics and
suitable tests of hypotheses have been employed aiming to identify nonlinear
processes \cite{Masson}. Nevertheless, the closed formula we found has other
possible uses, as for example the estimate of the process parameters since the
maximum likelihood estimation is computationally cumbersome within a
multivariate framework. For large samples or high frequency data, the closed
formulae of multivariate cumulants can be matched to multivariate
$k$-statistics, unbiased estimators with minimum variance
\cite{di2015symbolic}, and thus estimates of parameters can be recovered by an
analogous of GMM. This is in the agenda of our future research.

\end{document}